\documentclass[12pt,reqno]{amsart}
\usepackage{amscd,amsmath,amsthm,amssymb}
\usepackage{color}
\usepackage{tikz}
\usepackage{url}

 \newtheorem{Theorem}{Theorem}[section]
 
 \newtheorem{Corollary}[Theorem]{Corollary}
 \newtheorem{Proposition}[Theorem]{Proposition}

 \newtheorem{Definition}[Theorem]{Definition}

 \newtheorem{Conjecture}[Theorem]{Conjecture}
 
 \def\qed{\ifhmode\textqed\fi
       \ifmmode\ifinner\quad\qedsymbol\else\dispqed\fi\fi}
 \def\textqed{\unskip\nobreak\penalty50
        \hskip2em\hbox{}\nobreak\hfill\qedsymbol
        \parfillskip=0pt \finalhyphendemerits=0}
 \def\dispqed{\rlap{\qquad\qedsymbol}}

\def\ZZ{\mathbb{Z}}

\def\supp{\textup{supp}}

\def\Ind{\textup{Ind}}
\def\Del{\textup{Del}}
\def\Lk{\textup{Lk}}
\def\sort{\textup{sort}}
\def\ini{\textup{in}}
\def\p{\mathfrak{p}}
 
\begin{document}
 
\title{Sortable simplicial complexes and their associated toric rings}

\author{Antonino Ficarra, Somayeh Moradi}

\address{Antonino Ficarra, Departamento de Matem\'{a}tica, Escola de Ci\^{e}ncias e Tecnologia, Centro de Investiga\c{c}\~{a}o, Matem\'{a}tica e Aplica\c{c}\~{o}es, Instituto de Investiga\c{c}\~{a}o e Forma\c{c}\~{a}o Avan\c{c}ada, Universidade de \'{E}vora, Rua Rom\~{a}o Ramalho, 59, P--7000--671 \'{E}vora, Portugal}
\email{antonino.ficarra@uevora.pt}\email{antficarra@unime.it}

\address{Somayeh Moradi, Department of Mathematics, Faculty of Science, Ilam University, P.O.Box 69315-516, Ilam, Iran}
\email{so.moradi@ilam.ac.ir}


\begin{abstract}
Let $\Gamma$ be a $d$-flag sortable simplicial complex. We consider the toric ring $R_{\Gamma}=K[{\bf x}_Ft:F\in \Gamma]$ and the Rees algebra of the facet ideals $I(\Gamma^{[i]})$ of pure skeletons of $\Gamma$. We show that these algebras are Koszul, normal Cohen-Macaulay domains. Moreover, we study the Gorenstein property, the canonical module, and the $a$-invariant of the normal domain $R_{\Gamma}$ by investigating its divisor class group. Finally, it is shown that any $d$-flag sortable simplicial complex is vertex decomposable, which provides a  characterization of the Cohen-Macaulay property of such complexes.
\end{abstract}

\maketitle
\section*{Introduction}

The sortability concept was first defined by Sturmfels~\cite{St} for a finite set of monomials generated in one degree in a polynomial ring $S=K[x_1,\ldots,x_n]$ over a field $K$. This concept was extended by Herzog, Hibi and the second author of this paper~\cite{HHM} to a finite set of monomials of arbitrary degrees in $S$. Toric rings generated by sortable sets of monomials  have the desirable property that their defining ideal possesses a quadratic reduced Gröbner basis consisting of binomials that arise from sorting relations, as shown in~\cite[Theorem 1.4]{HHM} and ~\cite[Theorem 6.16]{EH}. 
In~\cite{HKMR}, the notion of sortability was applied to simplicial complexes. To each face $F$ of a simplicial complex $\Gamma$, one can associate a monomial ${\bf x}_F=\prod_{i\in F}x_i$. A simplicial complex $\Gamma$ is called {\em sortable} if the set of monomials ${\bf x}_F$ corresponding to the faces $F\in \Gamma$ forms a sortable set of monomials. In the mentioned paper, the Rees algebra of the $t$-independence ideal $I_t(G)$ of a proper interval graph $G$ was shown to be Koszul and a normal Cohen-Macaulay ring. The ideal $I_t(G)$ is the facet ideal of the $t$-pure skeleton the independence complex $\Delta_G$ of $G$, which is a sortable simplicial complex in this case~\cite[Theorem 8]{HKMR}. 

In this paper, we investigate more generally, how the  sortability property for a simplicial complex $\Gamma$ implies desirable algebraic properties for toric and Rees algebras attached to $\Gamma$ and for the Stanley-Reisner ring of $\Gamma$. 
We consider $d$-flag sortable simplicial complexes. A {\em $d$-flag} simplicial complex is a simplicial complex whose all minimal non-faces have cardinality $d$. This family of simplicial complexes turns out to exhibit interesting algebraic properties. The key reason for this, is that any $d$-flag sortable simplicial complex is the independence complex $\Ind(\Delta)$ of a so-called unit-interval simplicial complex $\Delta$. Roughly speaking, a unit-interval simplicial complex is a pure simplicial complex whose facets are obtained by taking the sets of facets of pure skeletons of some simplices that admit a suitable labeling. This fact is formally established in Theorem~\ref{Thm:Ind}.
For a $d$-flag sortable simplicial complex $\Gamma$ or equivalently for $\Gamma=\Ind(\Delta)$, where $\Delta$ is a unit-interval simplicial complex of dimension $d-1$, we mainly study

\begin{enumerate}
    \item[(a)] The Rees algebra $\mathcal{R}(I(\Gamma)^{[t]})$ of the facet ideal of pure skeletons $\Gamma^{[t]}$ of $\Gamma$. 
    \item[(b)] The divisor class group $\textup{Cl}(R_\Gamma)$, the Gorenstein property, the canonical module $\omega_{R_{\Gamma}}$ and the $a$-invariant  of the toric ring $R_{\Gamma}=K[{\bf x}_Ft:F\in \Gamma]$.
    \item[(c)] The Cohen-Macaulayness of $\Gamma$.
\end{enumerate}

We address (a) in Section \ref{sec:1}. To this aim, Theorem~\ref{Thm:Ind} plays a crucial role.
In Proposition~\ref{lcondition}, we show that the sortable ideal $I(\Gamma)^{[t]}$ satisfies the $\ell$-exchange property. Using a description of the reduced Gröbner basis of the defining ideal of Rees algebras of monomial ideals with the $\ell$-exchange property, as given by Herzog, Hibi and Vladoiu in~\cite{HHV} (see Theorem~\ref{ReesQuadratic}), we conclude that the Rees algebra $\mathcal{R}(I(\Gamma)^{[t]})$ has a quadratic reduced Gröbner basis. From this fact, we obtain Corollary~\ref{finally}, which states that this ring is a Koszul and normal Cohen-Macaulay domain, the ideal $I(\Gamma)^{[t]}$ satisfies the strong persistence property and all of its powers have linear resolutions. 

In Sections~2 and 3, we study the  Gorenstein property, the $a$-invariant and the canonical module of the toric ring $R_{\Gamma}$, via its divisor class group. The family of toric rings attached to  simplicial complexes were introduced first in~\cite{HMQ} and further studied in~\cite{FHS}, where their divisor class group was investigated.

When $\Gamma$ is flag, i.e., $2$-flag, $\Delta$ may be viewed as a graph $G$ and the ring $R_{\Gamma}$ is the toric ring of the stable set polytope of $G$. Moreover, the sortability of $\Gamma$ is equivalent to  $G$ being a proper interval graph. Notice that any proper interval graph is perfect. For a perfect $G$  and $\Gamma=\Ind(G)$, in~\cite{HO1} it was shown that $R_{\Gamma}$ is a normal ring. In Section 2, we consider more generally such flag complexes. For $\Gamma=\Ind(G)$ with $G$ a perfect graph, the divisor class group and the set of height one prime ideals of $R_{\Gamma}$ which determine the canonical module $\omega_{R_\Gamma}$ were determined in~\cite[Theorem 1.10, Corollary 1.12]{HMQ}. Using those descriptions, in Proposition~\ref{Prop:a-inv} we show that the $a$-invariant of $R_{\Gamma}$ is equal to $-(\omega(G)+1)$, where $\omega(G)$ denotes the clique number of $G$. 
In Theorem~\ref{Thm:IndG-perfect}, we recover a result by Hibi and Ohsugi~\cite{HO} which gives a combinatorial characterization for the Gorenstein property of $R_{\Gamma}$, when $\Gamma$ is the independence complex of a perfect graph. We provide a new short proof based on divisorial computations developed in \cite{HMQ} and \cite{FHS}. A fundamental fact  employed in our arguments is that a normal Cohen-Macaulay domain $R$, with the canonical module $\omega_{R}$, is Gorenstein if and only if the so-called \textit{canonical class} $[\omega_{R}]$ is zero in the divisor class group $\textup{Cl}(R)$. These results are applied to $\Gamma=\Ind(\Delta)$ for a $1$-dimensional unit-interval simplicial complex $\Delta$ in Corollary~\ref{a-invOneDim}.  

In Section 3, we consider $d$-flag sortable sortable simplicial complexes for $d>2$.  In other words, $\Gamma=\Ind(\Delta)$ for a unit-interval simplicial complex $\Delta$ with $\dim(\Delta)=d-1>1$. The set of height one monomial prime ideals of $R_{\Gamma}$ which do not contain the variable $t$, is known by~\cite[Proposition 1.9]{HMQ}. Our aim is to determine all height one monomial prime ideals containing $t$. We present a family of monomial prime ideals of height one which contain $t$ in Theorem~\ref{Thm:Lj} and Corollary~\ref{Cor:PGamma}. We conjecture that this family of prime ideals is precisely the set of height one monomial prime ideals of $R_{\Gamma}$ which contain $t$. In Proposition~\ref{disoint}, we prove this conjecture for a unit-interval simplicial complex $\Delta$ whose maximal cliques form a partition of the vertex set of $\Delta$. Knowing the subset of monomial prime ideals containing $(t)$ given in Corollary~\ref{Cor:PGamma} enables us to give a necessary condition for the Gorenstein property of $R_{\Gamma}$ in Proposition~\ref{Cor:PGamma1} and give an upper bound for the $a$-invariant of $R_{\Gamma}$ in Proposition~\ref{bound-a}. 
Moreover, in Corollary~\ref{Cor:radical} we determine when $(t)$ is a radical ideal, which leads to a simpler description for the canonical module of $R_{\Gamma}$.   

Finally, in Section~4 we consider interval simplicial complexes which extend the notion of unit-interval simplicial complexes to non-pure simplicial complexes. In Theorem~\ref{Thm:vd}, 
we show that for any interval simplicial complex $\Delta$, the independence complex $\Ind(\Delta)$ is vertex decomposable. This implies that  a $d$-flag sortable simplicial complex is Cohen-Macaulay if and only if it is pure.

\section{Characterization of  $d$-flag sortable simplicial complexes and their associated Rees algebras}\label{sec:1}

In this section, we study Rees algebras and ideals associated to $d$-flag sortable simplicial complexes. To this end, we begin by providing a characterization of these complexes, which is equivalent to describing pure $(d-1)$-dimensional simplicial complexes $\Delta$ for which the independence complex of $\Delta$ is sortable.

We begin by recalling relevant notation and   definitions.
Hereafter, $\Delta$ is a simplicial complex on the vertex set $[n]=\{1,2,\dots,n\}$ and $S=K[x_1,\ldots,x_n]$ is the polynomial ring over a field $K$. We say that a face $F\in\Delta$ has dimension $d$, if $|F|=d+1$. The {\em dimension} of $\Delta$ is defined as the maximum dimension of the faces of $\Delta$ and is denoted by $\dim(\Delta)$.  A \textit{facet} of $\Delta$ is a maximal face of $\Delta$ with respect to inclusion. We denote by $\mathcal{F}(\Delta)$, the set of all facets of $\Delta$, and we say that $\Delta$ is \textit{pure} if all facets of $\Delta$ have the same dimension. 
Furthermore, the {\em facet ideal} of $\Delta$ is defined as $I(\Delta)=(\textbf{x}_F: F\in \mathcal{F}(\Delta))$.

Now, let $\Delta$ be a pure $(d-1)$-dimensional simplicial complex $\Delta$. A subset $C\subseteq[n]$ is called a \textit{clique} of $\Delta$, if all $d$-subsets of $C$ are facets of $\Delta$.
A subset $D\subseteq[n]$ is called an \textit{independent set} of $\Delta$, if $D$ contains no facet of $\Delta$. The collection of all independent sets of $\Delta$ is a simplicial complex, called the \textit{independence complex} of $\Delta$, and denoted by $\Ind(\Delta)$.

We recall the concept of \textit{sortability} from \cite{HHM}. Let $u$ and $v$ be two monomials of $S$, and write
$$
uv=x_{i_1}x_{i_2}\cdots x_{i_{r}} \quad \text{with}\quad i_1\leq i_2\leq \cdots \leq i_{r}.
$$
We define the \textit{sorting} of the pair $(u,v)$ as
\[
\sort(u,v)=(u',v'),\quad \text{with}\quad u'=\prod_{j\in F}x_{i_j},\ \ \ v'=\prod_{j\in G}x_{i_j},
\]
where $F=\{j:\ 1\leq j\leq r,\ j\textrm{ is odd}\}$ and $G=\{j:\ 1\leq j\leq r,\ j\textrm{ is even}\}$. If $\sort(u,v)=(u,v)$, the pair $(u,v)$ is called a \textit{sorted pair}, and otherwise it is called an \textit{unsorted pair}. A finite set of monomials $\mathcal{M}\subset S$ is called \textit{sortable}, if $\sort(u,v)\in\mathcal{M}\times\mathcal{M}$ for all $u,v\in\mathcal{M}$. 

Now, let $\Delta$ be a simplicial complex. We say that $\Delta$ is \textit{sortable}, if the set of monomials $\mathcal{M}_\Delta=\{{\bf x}_F:F\in\Delta\}$ 
is sortable. Here ${\bf x}_F=\prod_{i\in F}x_i$ for a non-empty subset $F$ of $[n]$ and ${\bf x}_\emptyset=1$. Given $F,G\in\Delta$, we set $\sort(F,G)=(F',G')$, where the sets $F',G'$ are defined by the equation $\sort({\bf x}_F,{\bf x}_G)=({\bf x}_{F'},{\bf x}_{G'})$.


Given positive integers $i\le j$, we denote by $[i,j]$ the interval $\{i,i+1,\dots,j\}$. Following \cite{AV} and \cite{BSV},  a simplicial complex $\Delta$ of dimension $d-1$ is called a \textit{unit-interval simplicial complex} if $\Delta$ is pure, and for any facet $F=\{i_1<\dots<i_{d}\}\in\Delta$, the interval $[i_1,i_{d}]$ is a clique of $\Delta$. For such a simplicial complex, the maximal cliques of $\Delta$ are intervals. 

The following theorem gives a characterization of $d$-flag sortable simplicial complexes.

\begin{Theorem}\label{Thm:Ind}
Let $\Delta$ be a pure $(d-1)$-dimensional simplicial complex. Then $\Ind(\Delta)$ is sortable if and only if $\Delta$ is a unit-interval simplicial complex.
\end{Theorem}
\begin{proof}
Suppose that $\Ind(\Delta)$ is sortable. Let $F=\{i_1<i_2<\dots<i_{d}\}$ be a facet of $\Delta$. We need to show that the interval $B=[i_1,i_{d}]$ is a clique of $\Delta$. Let $\ell\in B\setminus F$. We have $i_{j-1}<\ell<i_j$ for some $2\le j\le d$. Consider the $d$-subsets of $B$ defined as follows,
\begin{align*}
    G\ &=\ \{i_1<\dots<i_{j-1}<\ell<i_{j+1}<\dots<i_{d}\},\\
    H\ &=\ \{i_1<\dots<i_{j-2}<\ell<i_{j}<i_{j+1}<\dots<i_{d}\}.
\end{align*}
We say that $G$ and $H$ are obtained from $F$ by a \textit{1-step exchange}. It is clear that any $d$-subset of $B$ can be obtained from $F$ by performing finitely many 1-step exchanges. We claim that $G$ and $H$ are again facets of $\Delta$. Then, iterations of 1-step exchanges yield that $B$ is indeed a clique of $\Delta$.

Assume by contradiction that $G\notin\mathcal{F}(\Delta)$. Then $G\in\Ind(\Delta)$. Furthermore, the set $F_1=F\setminus\{i_{j-1}\}$ is also independent, for it has size $d-1$. Since $\Ind(\Delta)$ is sortable, it follows that $\sort(F_1,G)\in\Ind(\Delta)\times\Ind(\Delta)$. However, $\sort(F_1,G)=(F,G_1)$, where $G_1=G\setminus\{i_{j-1}\}$. This is absurd, since $F\in\mathcal{F}(\Delta)$. Hence, we have $G\in\mathcal{F}(\Delta)$. Similarly, one can show that $H$ is a facet of $\Delta$.

Conversely, suppose that $\Delta$ is a unit-interval simplicial complex of dimension $d-1$. We show that $\Ind(\Delta)$ is sortable. Let $F,G\in\Ind(\Delta)$ with $|F|=r$ and $|G|=s$, and write ${\bf x}_F{\bf x}_G=x_{i_1}x_{i_2}\cdots x_{i_{r+s}}$ with $i_1\leq i_2\leq \cdots \leq i_{r+s}$. We need to show that $\sort(F,G)=(F',G')\in\Ind(\Delta)\times\Ind(\Delta)$, where $F'=\{i_k:\ 1\leq k\leq r+s, k\textrm{ is odd}\}$ and $G'=\{i_k:\ 1\leq k\leq r+s, k\textrm{ is even}\}$. Notice that if $r+s\le 2(d-1)$, then $|F'|\le d-1$ and $|G'|\le d-1$. So both $F'$ and $G'$ are independent, that is $(F',G')\in\Ind(\Delta)\times\Ind(\Delta)$. Now, suppose that $r+s\ge 2d-1$, and by contradiction assume that at least one of $F'$ and $G'$ is not an independent set of $\Delta$, say $F'$. Then $F'$ contains a facet $H=\{j_1<j_2<\dots<j_{d}\}\in\Delta$. By our assumption, the interval $B=[j_1,j_{d}]$ is a clique of $\Delta$. Notice that each $j_p$ is equal to $i_{2q-1}$ for some $q$. Therefore, the set $A=\{k:\,j_1\le i_k\le j_{d}\}$ contains at least $2d-1$ elements.
Since $F$ and $G$ are independent sets of $\Delta$ and $B$ is a clique of $\Delta$, we have $|B\cap F|\le d-1$ and $|B\cap G|\le d-1$.  These inequalities imply that $|A|\le 2d-2$, which is absurd.
\end{proof}

Given a simplicial complex $\Gamma$, let $R_\Gamma=K[{\bf x}_Ft:F\in\Gamma]\subset S[t]$.  For an integer $i\geq 0$, we denote by $\Gamma^{(i)}=\{F\in\Gamma:\dim F\le i\}$ the \textit{$i$th skeleton} of $\Gamma$, and by $\Gamma^{[i]}=\langle F\in\Gamma:\dim F=i\rangle$ the \textit{$i$th pure skeleton} of $\Gamma$. 
In \cite[Corollary 13]{HKMR}, it was shown that if $\Gamma$ is sortable, then $R_{\Gamma^{[i]}}$ is a Koszul and normal Cohen-Macaulay ring. However, we have more generally,

\begin{Proposition}
Let $\Gamma$ be a sortable simplicial complex.
Then $R_\Gamma$, $R_{\Gamma^{(i)}}$ and $R_{\Gamma^{[i]}}$ are Koszul, normal Cohen-Macaulay domains.
In particular, this holds when $\Gamma=\Ind(\Delta)$ and $\Delta$ is a unit-interval simplicial complex. 
\end{Proposition}
\begin{proof}
Note that if $\Gamma$ is sortable, then all its (pure) skeletons $\Gamma^{(i)}$ and $\Gamma^{[i]}$ are sortable. Therefore, by \cite[Theorem 1.4]{HHM}, there exist the so-called {\em sorting monomial orders} $<$ on the polynomial rings $K[y_F: F\in\Gamma^{(i)}]$ and $K[y_F: F\in\Gamma^{[i]}]$. So the defining ideal of the toric algebras  $R_{\Gamma^{(i)}}=K[{\bf x}_Ft:F\in\Gamma^{(i)}]$ and $R_{\Gamma^{[i]}}=K[{\bf x}_Ft:F\in\Gamma^{[i]}]$ have  quadratic reduced Gröbner bases with respect to the sorting order, consisting of binomials obtained from the unsorted pairs of monomials in $I(\Gamma^{(i)})$ and $I(\Gamma^{[i]})$, respectively. It follows that these rings are  Koszul, normal Cohen-Macaulay domains, see~\cite[Theorem 2.28, Corollary 4.26]{HHO} and \cite[Theorem 1]{Ho}.

The second statement follows from Theorem \ref{Thm:Ind} and the first statement.
\end{proof}

Next, we study the Rees algebra $\mathcal{R}(I(\Gamma^{[t]}))$ for a sortable simplicial complex $\Gamma$ and a positive integer $t$. To this aim we need to recall the concept of $\ell$-exchange property. The \textit{support} of a monomial $u\in S$ is the set $\supp(u)=\{x_i:\ x_i\ \textup{divides}\ u\}$.

Let $I\subset S$ be an equigenerated monomial ideal, and let $A=K[u:\ u\in \mathcal{G}(I)]$ be the toric algebra attached to $I$. Here $\mathcal{G}(I)$ denotes the minimal monomial generating set of $I$. Then we may write $A\cong R/L$, where $R=K[y_u:\ u\in \mathcal{G}(I)]$ is the polynomial ring and $L$ is the kernel of the $K$-algebra homomorphism $R\to A$ with $y_u\mapsto u$ for any $u\in \mathcal{G}(I)$. We fix a monomial order $<$ on $R$. A monomial $w \in R$  is called a {\em standard monomial} of $L$ with respect to $<$, if  $w\not \in \ini_{<}(L)$.

The concept of $\ell$-exchange
property was defined in~\cite{HHV}, as follows.

\begin{Definition}
\label{l}
{\em Keeping the above notation, we say that $I$ satisfies the \textit{$\ell$-exchange
property} with respect to the monomial order $<$ on $R$, if the following condition is satisfied: let $y_{u_1}\cdots y_{u_N}$ and $y_{v_1}\cdots y_{v_N}$ be two standard monomials of $L$ with respect to $<$  such that
\begin{enumerate}
\item[(i)]
$\deg_{x_r} (u_1\cdots u_N)=\deg_{x_r} (v_1\cdots v_N)$ for $r=1,\ldots,q-1$ with $q\leq n-1$,
\item [ (ii)]
$\deg_{x_q}(u_1\cdots u_N)<\deg_{x_q}(v_1\cdots v_N)$.
\end{enumerate}
Then there exists an integer $k$,  and an integer $q<j\leq n$ with $x_j\in\supp(u_k)$ such that  $x_qu_k/x_j\in I$.}
\end{Definition}

\begin{Proposition}
\label{lcondition}
Let $\Delta$ be a unit-interval simplicial complex on the vertex set $[n]$, and let $\Gamma=\Ind(\Delta)$. Then for all $t\geq 1$,  the ideal $I(\Gamma^{[t]})$ satisfies the $\ell$-exchange property with respect to the sorting order $<$.
\end{Proposition}
\begin{proof} Let $I=I(\Gamma^{[t]})$, and let $y_{u_1}\cdots y_{u_N}$ and $y_{v_1}\cdots y_{v_N}$ be standard  monomials of $R$ with respect to the sorting order $<$ and satisfying (i) and (ii) of Definition~\ref{l}. 
Since the non-sorted pairs correspond to initial terms in the sorting order, and $y_{u_1}\cdots y_{u_N}$ and $y_{v_1}\cdots y_{v_N}$ do not belong to $\ini_{<}(L)$, we conclude that all the pairs $(u_i,u_j)$ and $(v_i,v_j)$ are sorted for any $i<j$. Let $u_j=x_{i_{j,1}}\cdots x_{i_{j,t+1}}$ and $v_j=x_{i'_{j,1}}\cdots x_{i'_{j,t+1}}$ for any $1\leq j\leq N$. Then by \cite[Relation (6.3)]{EH},
\[
i_{1,1}\leq i_{2,1}\leq\!\cdots\!\leq\ i_{N,1}\leq i_{1,2}\leq i_{2,2}\leq\!\cdots\!\leq i_{N,2} \leq\!\cdots\!\leq i_{1,t+1} \leq i_{2,t+1} \leq\!\cdots\!\leq i_{N,t+1}
\]
and
\[
i'_{1,1}\leq i'_{2,1}\leq\!\cdots\!\leq\ i'_{N,1}\leq i'_{1,2}\leq i'_{2,2}\leq\!\cdots\!\leq i'_{N,2} \leq\!\cdots\!\leq i'_{1,t+1} \leq i'_{2,t+1} \leq\!\cdots\!\leq i'_{N,t+1}.
\]
By assumption $\deg_{x_r} (u_1\cdots u_N)=\deg_{x_r} (v_1\cdots v_N)$ for $r=1,\ldots,q-1$ with $q\leq n-1$. So from the above inequalities we obtain $i_{j,k}=i'_{j,k}$  for any $i_{j,k}\leq q-1$. Hence
$\deg_{x_r}(u_j)=\deg_{x_r}(v_j)$ for  all $1\leq j\leq N$ and $1\leq r\leq q-1$. By our assumption there exists an integer $1\leq m\leq N$ such that $\deg_{x_q}(u_m)<\deg_{x_q}(v_m)$.

Let $u_m=x_{k_1}x_{k_2}\cdots x_{k_t}$, $v_m=x_{\ell_1}x_{\ell_2}\cdots x_{\ell_t}$ such that $k_1<\cdots< k_t$ and $\ell_1<\cdots<\ell_t$ and $q=\ell_i$ for some $1\leq i<t$. Then $k_1=\ell_1,\ldots,k_{i-1}=\ell_{i-1}$ and $k_i>\ell_i=q$. We show that $x_{\ell_i}u_m/x_{k_i}\in I$. Suppose that this is not the case. Then $(\supp(u_m)\setminus \{x_{k_i}\})\cup\{x_{\ell_i}\}$ contains a facet $F$ of $\Delta$. Obviously, $x_{\ell_i}\in F$. We claim that $x_{k_s}\in F$ for some $s>i$. Otherwise, since $k_1=\ell_1,\ldots,k_{i-1}=\ell_{i-1}$, we get $F\subseteq\{x_{\ell_1},\ldots,x_{\ell_{i-1}},x_{\ell_i}\}\subseteq\supp(v_m)$, which contradicts to $v_m\in I$. So we have $x_{k_s}\in F$ for some $s>i$. Then, we have $\ell_i<k_i<k_s$. Since $x_{\ell_i},x_{k_s}\in F$ and $\Delta$ is a unit-interval simplicial complex, in both cases it follows that $F'=(F\setminus\{\ell_i\})\cup\{k_i\}\in\Delta$. Notice that $F'\subseteq\supp(u_m)$, with $|F'|=|F|$. Since $\Delta$ is pure, these imply that $F'$ is a facet of $\Delta$, and $\supp(u_m)\notin\Gamma$. So we get $u_m\notin I(\Gamma^{[t]})$, which is a contradiction. Thus $x_{\ell_i}u_m/x_{k_i}\in I$, as desired.
\end{proof}

The next result follows by combining \cite[Theorem 5.1]{HHV} with \cite[Theorem 6.16]{EH}.

\begin{Theorem}\label{ReesQuadratic}
Let $I\subset S$ be an equigenerated monomial ideal which is sortable and satisfies the $\ell$-exchange property with respect to the sorting order. Then there exists a monomial order $<$ on the polynomial ring $T=S[y_u: u\in \mathcal{G}(I)]$ such that the reduced Gr\"obner basis of the defining ideal $J\subset T$ of the Rees algebra $\mathcal{R}(I)$  with respect to $<$ is  quadratic.
\end{Theorem}

The next corollary generalizes \cite[Corollary 16]{HKMR}, where the following result was shown for flag ($2$-flag) sortable complexes.

\begin{Corollary}
\label{finally}
Let  $\Gamma$ be a $d$-flag sortable simplicial complex. Then
\begin{enumerate}
\item [\textup{(a)}] $\mathcal{R}(I(\Gamma^{[t]}))$ is a Koszul, normal Cohen-Macaulay domain. 
\item[\textup{(b)}] All powers of $I(\Gamma^{[t]})$ have linear resolutions. 
\item[\textup{(c)}] $I(\Gamma^{[t]})$ satisfies the strong persistence property.
\end{enumerate}
\end{Corollary}
\begin{proof}
Since $\Gamma$ is $d$-flag, we have $\Gamma=\Ind(\Delta)$, where $\Delta$ is 
a pure simplicial complex of dimension $d-1$. By Theorem~\ref{Thm:Ind}, the sortability of $\Gamma$ implies that $\Delta$ is a unit-interval simplicial complex. So by  Proposition~\ref{lcondition}, $I(\Gamma^{[t]})$  satisfies the $\ell$-exchange
property with respect to the sorting order $<$. Since $I(\Gamma^{[t]})$ is a sortable monomial ideal, by Theorem~\ref{ReesQuadratic}, the defining ideal of the Rees algebra $\mathcal{R}(I(\Gamma^{[t]}))$ has a quadratic reduced Gr\"obner basis. Using this fact, then (a), (b) and (c) follow respectively from \cite[Theorem 2.28, Corollary 4.26]{HHO}, \cite[Corollary 1.2]{HHZ} and \cite[Corollary 1.6]{HQ}.
\end{proof}

\section{The toric ring of independence complex of perfect graphs}\label{sec:2}

Let $\Gamma$ be a $d$-flag sortable simplicial complex.  We study the Gorenstein property and the $a$-invariant of the normal toric ring $R_\Gamma$ by investigating the divisor class group $\textup{Cl}(R_\Gamma)$ of this ring.
As discussed in Section~\ref{sec:1}, $R_\Gamma$ is normal and $\Gamma=\Ind(\Delta)$, where $\Delta$ is a unit-interval simplicial complex of dimension $d-1$.
In this section, we consider the case $d=2$. The case $d>2$ will be discussed in Section~\ref{sec:3}.

When $d=2$, we may view $\Delta$ as a graph whose edges are the facets of $\Delta$. Then $\Gamma$ is the independence complex of a proper interval graph $G$, see \cite{HKMR}. In this section, we consider more generally the family of perfect graphs, which indeed includes proper interval graphs. Recall that a graph $G$ is called a {\em perfect} graph, if $G$ and $G^c$ do not contain induced odd cycles of length $r>3$. Here, by $G^c$ we denote the \textit{complementary graph} of $G$. That is $V(G^c)=V(G)$ and the edges of $G^c$ are the non-edges of $G$.

As was mentioned in the introduction, when $\Gamma=\Ind(G)$ and $G$ is a perfect graph, the ring $R_{\Gamma}$ is a normal Cohen-Macaulay domain. Moreover, the divisor class group and the set of height one prime ideals of $R_{\Gamma}$ which determine the canonical module $\omega_{R_\Gamma}$ of $R_\Gamma$ were described in~\cite[Theorem 1.10, Corollary 1.12]{HMQ}.

The $K$-algebra $R_{\Gamma}$ is standard graded if we put $\deg(x_1^{a_1}\cdots x_n^{a_n}t^k)=k$. Recall that the \textit{$a$-invariant} of a graded Cohen-Macaulay ring $R$ admitting a graded canonical module $\omega_R$ is defined as $a(R)=-\min\{j:(\omega_{R})_j\ne0\}$.

The {\em clique number} of a graph $G$ is defined as the maximum cardinality of cliques of $G$, and is denoted by $\omega(G)$.  

\begin{Proposition}\label{Prop:a-inv}
Let $G$ be a perfect graph and    $\Gamma=\Ind(G)$. Then $a(R_{\Gamma})=-(\omega(G)+1)$. In particular, $x_1\cdots x_nt^{\omega(G)+1}\in\omega_{R_\Gamma}$.
\end{Proposition}
\begin{proof}
Let $C_1,\ldots,C_r$ be the minimal vertex covers of the graph $G_{\Gamma}$ whose edge set is $\mathcal{F}(\Gamma^{[1]})$. By \cite[Theorem 1.10]{HMQ}, we have $\omega_{R_\Gamma}=P_{C_1}\cap\cdots\cap P_{C_r}\cap Q_1\cap\dots\cap Q_n$, where $P_{C_i}=({\bf x}_Ft: F\in \Gamma_{C_i})$ and $Q_i=({\bf x}_Ft:F\in\Gamma,i\in F)$.  
Here $\Gamma_{C_i}$ is the induced subcomplex of $\Gamma$ on the set $C_i$.

We set $u=x_1\cdots x_nt^{\omega(G)+1}$ and $p_u=(1,\ldots,1,\omega(G)+1)\in \mathbb{N}^{n+1}$.
Notice that $G_{\Gamma}=G^c$. Moreover, $C_i=[n]\setminus B_i$, where $B_i$ is a maximal independent set of $G_{\Gamma}=G^c$ for $1\leq i\leq r$. The maximal independent sets of $G^c$ are just the maximal cliques of $G$. Therefore, $B_i=[n]\setminus C_i$ for $i=1,\dots,r$ are the maximal cliques of $G$. By the proof of \cite[Theorem 1.3]{HMQ} the support form of $P_{C_i}$ is $f_i(x)=-\sum_{j\in B_i}x_j+x_{n+1}=0$ for $i=1,\dots,r$. Moreover, by \cite[Proposition 1.9]{HMQ}, the support form of $Q_j$ is $g_j(x)=x_j$ for all $j$.  First we show that $u\in\omega_{R_\Gamma}$.
Notice that $f_i(p_u)=-|B_i|+\omega(G)+1>0$, since $|B_i|\leq \omega(G)$ for all $i$. So $u\in P_{C_i}$ for all $i$. Moreover, $g_j(p_u)=1>0$ and hence $u\in Q_j$ for all $j$. Therefore, $u\in \omega_{R_\Gamma}$. This shows that
$a(R_{\Gamma})\geq -(\omega(G)+1)$.
Now, consider an element $v=x_1^{a_1}\cdots x_n^{a_n}t^{k}\in \omega_{R_\Gamma}$. Then for $p_v=(a_1,\ldots,a_n,k)$ we have $f_i(p_v)>0$ and $g_j(p_v)>0$ for all $i$ and $j$. It follows that $\sum_{j\in B_i}a_j<k$ for any $1\leq i\leq r$ and $a_j>0$ for all $j$. Hence $|B_i|\leq \sum_{j\in B_i}a_j<k$ for all $1\leq i\leq r$. This implies that $\omega(G)<k$. Hence $k\geq \omega(G)+1$. Therefore, $a(R_{\Gamma})\leq -(\omega(G)+1)$. 
\end{proof}

Next, we provide a new proof of the Gorenstein characterization  of $R_\Gamma$ given in \cite[Theorem 2.1(b)]{HO} in a simpler fashion. 

\begin{Theorem}\label{Thm:IndG-perfect}
    Let $G$ be a perfect graph and  $\Gamma=\Ind(G)$. The following conditions are equivalent.
    \begin{enumerate}
        \item[\textup{(a)}] $R_\Gamma$ is Gorenstein.
        \item[\textup{(b)}] All the maximal cliques of $G$ have the same cardinality.
        \item[\textup{(c)}] $G^c$ is unmixed.
    \end{enumerate}
    Moreover, if the equivalent conditions hold, then $$\omega_{R_\Gamma}=(x_1\cdots x_nt^{\omega(G)+1}).$$
\end{Theorem}
\begin{proof}
Let $C_1,\ldots,C_r$ be the minimal vertex covers of the graph $G^c$. 
As was shown in the proof of Proposition~\ref{Prop:a-inv}, $C_i=[n]\setminus B_i$ for all $i$, where $B_1,\ldots,B_r$ are the maximal cliques of $G$. The support form of $P_{C_i}$ is $f_i(x)=-\sum_{j\in B_i}x_j+x_{n+1}=0$ for $i=1,\dots,r$.
Now, by \cite[Theorem 4.3]{FHS}, $R_\Gamma$ is Gorenstein if and only if there exists an integer $a$ such that $1+|B_i|=a$ for all $i=1,\dots,r$. This implies that the statements (a) and (b) are equivalent. Since any minimal vertex cover of $G^c$ is of the form $[n]\setminus B_i$, we see that (b) is further equivalent to (c).

Finally, suppose that $R_\Gamma$ is Gorenstein. Then, the assertion about the canonical module follows from Proposition \ref{Prop:a-inv} and the fact that $\omega_{R_\Gamma}$ is a principal ideal generated in degree $-a(R_\Gamma)=\omega(G)+1$.
\end{proof}

We apply these results to a 1-dimensional unit-interval simplicial complex $\Delta$. Let $G_\Delta$ be the graph whose edges are the facets of $\Delta$. Since $G_\Delta$ is a proper interval graph, it is perfect. Hence, as a special case of Theorem \ref{Thm:IndG-perfect} we have

\begin{Corollary}\label{a-invOneDim}
Let $\Delta$ be a 1-dimensional unit-interval simplicial complex, and let $\Gamma=\Ind(\Delta)$. Then the $a$-invariant of $R_\Gamma$ is equal to
$$
-\max\{|B|+1: B \textrm { is a maximal clique of } \Delta\}.
$$
Moreover, $R_\Gamma$ is Gorenstein if and only if $G_\Delta^c$ is unmixed, and in this case
$$
\omega_{R_\Gamma}=(x_1\cdots x_nt^{\omega(G)+1}).
$$
\end{Corollary}

\section{The divisor class group of the toric ring of sortable simplicial complexes}\label{sec:3}

Throughout this section $\Gamma$ is a $d$-flag sortable simplicial complex with $d>2$. In other words, $\Gamma=\Ind(\Delta)$, where $\Delta$ is a unit-interval simplicial complex on $[n]$ with $\dim(\Delta)=d-1>1$.
We investigate the divisor class group $\textup{Cl}(R_{\Gamma})$, the canonical module $\omega_{R_{\Gamma}}$, and the Gorenstein property of $R_{\Gamma}$.

By \cite[Theorem 1.1]{HMQ}, the minimal prime ideals of $(t)\subset R_{\Gamma}$ determine
$\textup{Cl}(R_{\Gamma})$. Moreover, from \cite[Proposition 1.4]{HMQ} we know that for any minimal vertex cover $C$ of the graph $G_\Gamma$ whose edge are the facets of $\Gamma^{[1]}$, the ideal $P_{C}=({\bf x}_Ft: F\in \Gamma_{C})$ is a minimal prime ideal of $(t)$, where  $\Gamma_{C}$ is the induced subcomplex of $\Gamma$ on the set $C$.
Since by assumption $d>2$, it follows that $G_\Gamma$ is the complete graph on the vertex set $[n]$. Hence $\{[n]\setminus\{i\}:i\in[n]\}$ is the set of minimal vertex covers of $G_\Gamma$. Let $C_i=[n]\setminus\{i\}$ for all $i$. We denote the prime ideal $P_{C_i}=({\bf x}_Ft: F\in \Gamma_{C_i})$ by $P_i$. Hence, $P_1,\ldots,P_n$ are among the minimal prime ideals of $(t)$. 

The next result presents another family of minimal prime ideals of $(t)$.
Let $B_1,\dots,B_m$ be the maximal cliques of $\Delta$, which are indeed intervals. For each $j\in[m]$, we consider the monomial ideal
$$
L_j\ =\ ({\bf x}_Ft\ :\ F\in\Gamma\ \text{and}\ f_j(p_F)>0),
$$
where $$f_j(x)=-(\sum_{k\in B_j}x_k)+(d-1)x_{n+1}$$ and $p_F=\sum_{i\in F}e_i+e_{n+1}\in\ZZ^{n+1}$. Here, $e_1,\dots,e_{n+1}$ is the canonical basis of $\ZZ^{n+1}$.

\begin{Theorem}\label{Thm:Lj}
    With the notation introduced, $L_j$ is a minimal monomial prime ideal of $(t)\subset R_\Gamma$, for all $j\in[m]$.
\end{Theorem}
\begin{proof}
    Firstly, we prove that $L_j$ is a prime ideal of $R_\Gamma$. To this end, it is enough to show that the hyperplane $H_j$ defined by $f_j$ is a supporting hyperplane of   $\mathbb{R}_+ A_\Gamma$, where $A_\Gamma$ is the affine semigroup generated by the lattice points $p_F$ with $F\in\Gamma$ and $\mathbb{R}_+ A_\Gamma$ is the smallest cone containing $A_\Gamma$. To do so, let $F\in\Gamma$. Since $B_j$ is a clique of $\Delta$ and $F$ is an independent set of $\Delta$, it follows that $F$ contains at most $d-1$ elements from $B_j$. Hence, $f_j(p_F)\ge0$, and $f_j(p_F)=0$ if and only if $|F\cap B_j|=d-1$. This shows that $H_j\cap \mathbb{R}_+ A_\Gamma\ne\emptyset$ and $f_j(p_F)\ge0$ for all $F\in\Gamma$. Hence, $H_j$ is a supporting hyperplane of $\mathbb{R}_+ A_\Gamma$ and $L_j$ is a prime ideal containing $t$.

    Now, we show that $L_j$ is a minimal prime ideal of $(t)$. Let $P$ be a minimal prime ideal of $(t)$ such that $(t)\subseteq P\subseteq L_j$. We will prove that $P=L_j$. We proceed by induction on $i>0$ to show that ${\bf x}_Ft\in P$ for all $F\in\Gamma$ with $|F|=i$ such that $f_j(p_F)>0$.

    Let $i=1$. Notice that for any $F\in\Gamma$ with $|F|=1$, we have $f_j(p_F)>0$. So $x_kt\in L_j$ for all $k\in [n]$. By \cite[Lemma 1.2]{HMQ}, the set $C=\{k\in[n]:x_kt\in P\}$ is a vertex cover of $G_\Gamma$. Since $G_\Gamma$ is the complete graph, either $C=[n]$ or $C=[n]\setminus\{h\}$ for some $h\in[n]$. 
    Hence, we must show that $C=[n]$. Suppose by contradiction this is not the case. Then $C=[n]\setminus\{h\}$ and $x_ht\notin P$. Since $|B_j|\ge d$, we can find a $(d-1)$-subset of $B_j$, call it $F$, which does not contain $h$. Then $F\in\Gamma$ and $f_j(p_F)=0$. Hence ${\bf x}_Ft\notin L_j$, and so ${\bf x}_Ft\notin P$, too. Let $q\in F$. Then $F_1=(F\setminus\{q\})\cup\{h\}\in\Gamma$ because $|F_1|=d-1$. Therefore, $(x_ht)({\bf x}_Ft)=(x_qt)({\bf x}_{F_1}t)\in P$ because $q\in C$. This is a contradiction, because $x_ht,{\bf x}_Ft\notin P$. This shows that indeed $C=[n]$.

    Suppose now $i>1$. Let $F\in\Gamma$ be such that $|F|=i$ and $f_j(p_F)>0$. Then $|F\cap B_j|\le d-2$. We must show that ${\bf x}_Ft\in P$. We distinguish two cases.

    \textsc{Case 1}. Suppose that $F\subseteq B_j$. We show that ${\bf x}_Ft\in P$.  
    
    First suppose that there exists a subset $G\subset B_j$ of size $i$ with $G\neq F$ such that ${\bf x}_Gt\notin P$. Then, we can find $h\in F\setminus G$. Notice that $i=|F|=|F\cap B_j|\le d-2$. Hence $G_1=G\cup\{h\}\in\Gamma$ because $|G_1|=i+1\le d-1$ and so $G_1$ is an independent set of $\Delta$. Moreover, $F_1=F\setminus\{h\}\in\Gamma$ too, and $f_j(p_{F_1})>f_j(p_F)>0$. By our induction hypothesis, ${\bf x}_{F_1}t\in P$. Then, $({\bf x}_Ft)({\bf x}_Gt)=({\bf x}_{F_1}t)({\bf x}_{G_1}t)\in P$ because ${\bf x}_{F_1}t\in P$. Since ${\bf x}_Gt\notin P$, we conclude that ${\bf x}_Ft\in P$, as desired.

    Suppose now that for all subsets $H\subset B_j$ of size $i$ different from $F$ we have ${\bf x}_{H}t\in P$. Since $|B_j\setminus F|\ge d-i$, we can find elements $h_1,\dots,h_{d-i}$  in the set $B_j\setminus F$. Moreover, since $i>1$, we have $d-i\le d-2$. Therefore, we can find a subset $G\subset B_j$ of size $d-1$ that contains $h_1,\dots,h_{d-i}$. Then $G\in\Gamma$ and since $f_j(p_G)=0$, we have ${\bf x}_Gt\notin L_j$ and so ${\bf x}_Gt\notin P$ too. Set $F_1=F\cup\{h_1,\dots,h_{d-i-1}\}$ and $G_1=G\setminus\{h_1,\dots,h_{d-i-1}\}$. Then $|F_1|=d-1$ and so $F_1\in\Gamma$. Moreover, $|G_1|=i$ and $G_1\ne F$ because $h_{d-i}\in G_1\setminus F$. Therefore, ${\bf x}_{G_1}t\in P$ by our assumption. Hence, we have $({\bf x}_Ft)({\bf x}_Gt)=({\bf x}_{F_1}t)({\bf x}_{G_1}t)\in P$. Since $ {\bf x}_Gt\notin P$, we conclude that ${\bf x}_Ft\in P$, and this finishes the proof of this case.

    \textsc{Case 2.} Suppose now that $F\cap([n]\setminus B_j)\ne\emptyset$, and let $h\in F\cap([n]\setminus B_j)$. We define a subset $G$ of $B_j$ as follows,
    $$
    G\ =\ \begin{cases}
        [\max B_j-(d-2),\,\max B_j]&\textup{if}\ h<\min B_j,\\
        \hfil[\min B_j,\,\min B_j+(d-2)]&\textup{if}\ h>\max B_j.
    \end{cases}
    $$
    Notice that $|G|=d-1$ and $G\subset B_j$. Therefore, $f_j(p_G)=0$ and so ${\bf x}_Gt\notin L_j$. Thus ${\bf x}_Gt\notin P$. It is easily seen that $G_1=G\cup\{h\}$ is again an independent set of $\Delta$. Indeed, $G_1$ does not contain any $d$-subset of any of the intervals $B_1,\ldots,B_m$. Let $F_1=F\setminus\{h\}$. By induction on $i$, ${\bf x}_{F_1}t\in P$. Hence $({\bf x}_Ft)({\bf x}_Gt)=({\bf x}_{F_1}t)({\bf x}_{G_1}t)\in P$. Since ${\bf x}_Gt\notin P$, it follows that ${\bf x}_Ft\in P$, and this concludes the proof.
\end{proof}

\begin{Corollary}\label{Cor:PGamma}
    Let $\Delta$ be a unit-interval simplicial complex with $\dim(\Delta)>1$,  let $\Gamma=\Ind(\Delta)$, and let $\mathcal{P}_\Gamma$ be the set of the height one monomial prime ideals of $R_\Gamma$ containing $t$. With the notation introduced before, we have
    \begin{equation}\label{eq:inclusionPGamma}
        \{P_i:i=1,\dots,n\}\cup\{L_j:j=1,\dots,m\}\ \subseteq\ \mathcal{P}_\Gamma.
    \end{equation}
\end{Corollary}

Based on some computational evidence, we expect that the inclusion (\ref{eq:inclusionPGamma}) is in fact an equality, as stated in
\begin{Conjecture}\label{Conj:PGamma}
    Let $\Delta$ be a unit-interval simplicial complex with $\dim(\Delta)>1$, and let $\Gamma=\Ind(\Delta)$. The set of height one monomial prime ideals of $R_\Gamma$ containing $t$ is
    $$
    \mathcal{P}_\Gamma\ =\ \{P_i:i=1,\dots,n\}\cup\{L_j:j=1,\dots,m\}.
    $$
\end{Conjecture}

At present, we have not been able to establish this conjecture in full generality. However, we could prove it in the following situation.
\begin{Proposition}\label{disoint}
    Let $\Delta$ be a unit-interval simplicial complex with $\dim(\Delta)>1$, and let $\Gamma=\Ind(\Delta)$. Suppose that the maximal cliques of $\Delta$ form a partition of $V(\Delta)$. Then Conjecture \ref{Conj:PGamma} holds.
\end{Proposition}
\begin{proof}
    Let $B_1,\dots,B_m$ be the maximal cliques of $\Delta$. Let $A_\Gamma$ be the affine semigroup generated by the lattice points $p_F$ with $F\in\Gamma$. By \cite[Proposition 1.9]{HMQ}, the height one monomial prime ideals not containing $t$ are the ideals $Q_i=({\bf x}_Ft:F\in\Gamma,i\in F)$, for $i=1,\dots,n$. We know from the proofs of Theorem~\ref{Thm:Lj}, and \cite[Theorem 1.3, Proposition 1.9]{HMQ} that 
    \begin{align}
    \label{eq:fQ}f_{Q_i}(x)\ &=\ x_i,&&\textup{for}\ i=1,\dots,n,\\
    \label{eq:fP}f_{P_i}(x)\ &=\ -x_i+x_{n+1},&& \textup{for}\ i=1,\dots,n,\\
    \label{eq:fL}f_{L_j}(x)\ &=\ -(\sum_{i\in B_j}x_i)+(d-1)x_{n+1},&&\textup{for}\ j=1,\dots,m,
    \end{align}
    are the support forms of the prime ideals $Q_i$, $P_i$, for $i=1,\dots,n$, and the ideals $L_j$ for $j=1,\dots,m$, respectively. Let
    $$
    A\ =\ (\bigcap_{i=1}^n H_{Q_i}^{(+)})\cap(\bigcap_{i=1}^n H_{P_i}^{(+)})\cap(\bigcap_{j=1}^m H_{L_j}^{(+)})\cap\ZZ^{n+1},
    $$
    where each $H_{Q_i}^{(+)}$, $H_{P_i}^{(+)}$, $H_{L_j}^{(+)}$ is the half-space defined by the equations $f_{Q_i}(x)\ge0$, $f_{P_i}(x)\ge0$, and $f_{L_j}(x)\ge0$, respectively. By Corollary \ref{Cor:PGamma}, we have the inclusion $A_\Gamma\subseteq A$. Therefore, if we show the opposite inclusion, it will follow that the inclusion (\ref{eq:inclusionPGamma}) is an equality, and so Conjecture \ref{Conj:PGamma} holds.

    To prove that $A\subseteq A_\Gamma$, let $p=(a_1,\dots,a_n,k)\in A$ and $u=x_1^{a_1}\cdots x_n^{a_n}t^k$. It follows from equations (\ref{eq:fQ}), (\ref{eq:fP}) and (\ref{eq:fL}) that $0\le a_i\le k$ for all $i=1,\dots,n$ and $\sum_{i\in B_j}a_i\le (d-1)k$ for all $j=1,\dots,m$, where $d-1=\dim(\Delta)$. Notice that for $k=0$, we have $u=1$ and so $p\in A_\Gamma$. Now, let $k>0$. Proceeding by induction on $k\ge1$, we will show that $u\in R_\Gamma$ and so $p\in A_\Gamma$, as desired.

    For the base case, let $k=1$. Then, $0\le a_i\le 1$ for $i=1,\dots,n$, and so $u={\bf x}_Ft$ for some $F\subset[n]$. Since $\sum_{i\in B_j}a_i\le d-1$ for all $j=1,\dots,m$, we have $|F\cap B_j|\le d-1$. This shows that $F$ is an independent set of $\Gamma$, and so $u\in R_\Gamma$.

    Now, let $k>1$. We claim that $u=({\bf x}_Ft)v$ for some monomials ${\bf x}_Ft\in R_{\Gamma}$ and $v=x_1^{b_1}\cdots x_n^{b_n}t^{k-1}$ such that $q=(b_1,\dots,b_n,k-1)\in A$. Having this claim proved, it follows by induction that $v\in R_\Gamma$ and hence $u\in R_\Gamma$ too, as desired.

    To this end, for each $j$, we set $d_j=\sum_{i\in B_j}a_i$. Now fix $j\in\{1,\dots,m\}$. We will 
    choose an appropriate subset $F_j\subset B_j$ of size at most $d-1$.
    To this aim, let $h_j=\max\{0,d_j-(d-1)(k-1)\}$. If $h_j=0$, then $d_j\le (d-1)(k-1)$ and we put $F_j=\{i\in B_j:a_i=k\}$. Notice that $|F_j|\le d-1$. Otherwise, if $h_j>0$, first we claim that $|\{i\in B_j:a_i>0\}|\ge h_j$. Indeed, if such set has size strictly less than $h_j$, then $d_j<h_jk=[d_j-(d-1)(k-1)]k$. From this equation it would follow that $(d-1)(k-1)k<d_j(k-1)$. Since $k>1$, we would have $d_j>(d-1)k$ which is absurd. Hence $|\{i\in B_j:a_i>0\}|\ge h_j$. Notice that since $d_j\leq (d-1)k$, we have $h_j\leq d-1$. We put $F_j=\{i\in B_j:a_i=k\}\cup G_j$, where $G_j=\emptyset$ if $|F_j|\geq h_j$, or else we choose $G_j\subseteq\{i\in B_j:0<a_i<k\}$ to be a subset of size $h_j-|\{i\in B_j:a_i=k\}|>0$. 
    
    In any case $F_j$ is a subset of $B_j$ with $h_j\le |F_j|\le d-1$. Let $F=F_1\cup\dots\cup F_m$. Since $V(\Delta)=B_1\sqcup\cdots\sqcup B_m$, it follows that $F\in\Gamma$. We can write $u=({\bf x}_Ft)v$, with ${\bf x}_Ft\in R_\Gamma$ and $v=x_1^{b_1}\cdots x_{n}^{b_n}t^{k-1}$. Moreover, by the construction of the sets $F_j$, it follows that $0\le b_i\le k-1$ for all $i$, and $\sum_{i\in B_j}b_i=d_j-|F_j|\le d_j-h_j \le (d-1)(k-1)$ for all $j$. Hence, $(b_1,\dots,b_n,k-1)\in A$. By induction hypothesis on $k$, it follows that $v\in R_\Gamma$. Hence $u\in R_\Gamma$ too, as desired.
\end{proof}

Corollary~\ref{Cor:PGamma} allows us to give a necessary condition for the Gorenstein property of $R_\Gamma$ for any $d$-flag sortable simplicial complex with $d>2$, as follows. 

\begin{Proposition}\label{Cor:PGamma1}
    Let $\Delta$ be a unit-interval simplicial complex, and let $\Gamma=\Ind(\Delta)$.  If $R_\Gamma$ is Gorenstein and $\dim(\Delta)=d-1>1$, then $|B|=2d-3$ for all maximal cliques $B$ of $\Delta$. The converse holds if Conjecture~\ref{Conj:PGamma} holds.
\end{Proposition}
\begin{proof}
      Let $\mathcal{P}_\Gamma=\{\p_1,\dots,\p_r\}$, and let $f_i(x)=\sum_{j=1}^{n+1}c_{i,j}x_j$ be the support form associated with $\p_i$, for $i=1,\dots,r$. By \cite[Theorem 4.3]{FHS}, $R_\Gamma$ is Gorenstein if and only if there exists $a\in\ZZ$ such that $1-\sum_{j=1}^nc_{i,j}=ac_{i,n+1}$ for all $i=1,\dots,r$. By Corollary \ref{Cor:PGamma}, we know that the inclusion (\ref{eq:inclusionPGamma}) holds. Recalling that the support form of $P_i\in\mathcal{P}_\Gamma$ is $f(x)=-x_i+x_{n+1}$ we obtain that $a=2$. Whereas, letting $B_1,\dots,B_m$ be the maximal cliques of $\Delta$, from the fact that $L_j\in\mathcal{P}_\Gamma$ for all $j=1,\dots,m$, we deduce that $1+|B_j|=a(d-1)$ for all $j$. Therefore, we obtain that $|B|=2d-3$ for all maximal cliques $B$ of $\Delta$.

    Finally, if the inclusion (\ref{eq:inclusionPGamma}) is an equality, then the previous argument shows that $R_\Gamma$ is Gorenstein if and only if $|B|=2d-3$, for all maximal cliques $B$ of $\Delta$.
\end{proof}

For a pure simplicial complex $\Delta$, the \textit{clique number} of $\Delta$, denoted by $\omega(\Delta)$, is defined as the largest size of a clique of $\Delta$.
Analogous to Corollary \ref{a-invOneDim} we have

\begin{Proposition}\label{bound-a}
Let $\Delta$ be a unit-interval simplicial complex with $\dim(\Delta)=d-1>1$, and let $\Gamma=\Ind(\Delta)$. Then
\begin{equation}\label{eq:a-inv-bound}
    a(R_{\Gamma})\ \leq\ \begin{cases}
\,-\lceil\frac{\omega(\Delta)}{d-1}\rceil&\textit{if}\ d-1\ \textit{does not divide}\ \omega(\Delta),\\
\,-\lceil\frac{\omega(\Delta)}{d-1}\rceil-1&\textit{if}\ d-1\ \textit{divides}\ \omega(\Delta).
\end{cases}
\end{equation}
If Conjecture~\ref{Conj:PGamma} holds, then \textup{(\ref{eq:a-inv-bound})} becomes an equality. 
\end{Proposition}
\begin{proof}
Let $B_1,\dots,B_m$ be the maximal clique intervals of $\Delta$. By Corollary \ref{Cor:PGamma}, we know that the inclusion (\ref{eq:inclusionPGamma}) holds. By \cite[Proposition 1.9]{HMQ}, the height one monomial prime ideals not containing $t$ are the ideals $Q_i=({\bf x}_Ft:F\in\Gamma,i\in F)$, for $i=1,\dots,n$. Therefore, for any monomial $u=x_1^{a_1}\cdots x_n^{a_n}t^k\in\omega_{R_\Gamma}$, we have that $u\in(\bigcap_{i=1}^nQ_i)\cap(\bigcap_{i=1}^nP_i)\cap(\bigcap_{j=1}^mL_j)$. Then we should have $0<a_i<k$ for $i=1,\dots,n$ and $d_j=\sum_{i\in B_j}a_i<(d-1)k$ for $j=1,\dots,m$. Since each $a_i$ is positive, we obtain that $|B_j|\le d_j<(d-1)k$. Hence $k>|B_j|/(d-1)$ for all $j$, and so $k\ge \lceil\omega(\Delta)/(d-1)\rceil$. This implies that $a(R_\Gamma)\le-\lceil\omega(\Delta)/(d-1)\rceil$.

Suppose now, in addition, that $d-1$ divides $\omega(\Delta)$. Then, there exist integers $\ell\in [m]$ and $k>0$ such that $|B_{\ell}|=\omega(\Delta)=k(d-1)$. Notice that $k=\lceil\frac{\omega(\Delta)}{d-1}\rceil$. We claim that $a(R_\Gamma)<-k$. Suppose by contradiction that this is not the case. Then the previous argument shows that  $a(R_\Gamma)=-k$. So we could find a monomial $u=x_1^{a_1}\cdots x_n^{a_n}t^k\in\omega_{R_\Gamma}$ such that $0<a_i<k$ for $i=1,\dots,n$ and $d_j=\sum_{i\in B_j}a_i<(d-1)k$ for $j=1,\dots,m$. However, for $j=\ell$ we have $(d-1)k=|B_{\ell}|\le d_{\ell}<(d-1)k$, which is absurd.

Assume that the inclusion (\ref{eq:inclusionPGamma}) is an equality. We claim that $u=x_1\cdots x_nt^{k}\in\omega_{R_\Gamma}$ with $k=\lceil\frac{\omega(\Delta)}{d-1}\rceil$ if $d-1$ does not divide $\omega(\Delta)$, and $k=\lceil\frac{\omega(\Delta)}{d-1}\rceil+1$, otherwise. This will show that equality holds in (\ref{eq:a-inv-bound}). Since, by assumption (\ref{eq:inclusionPGamma}) is an equality, from the equations (\ref{eq:fQ}), (\ref{eq:fP}) and (\ref{eq:fL}), it follows that a monomial $v=x_1^{b_1}\cdots x_n^{b_n}t^\ell$ belongs to $\omega_{R_\Gamma}$ if and only if $0<b_i<\ell$ and $\sum_{i\in B_j}b_i<(d-1)\ell$ for all $i$ and $j$. Since $d-1>1$, any interval $B_j$ has size at least $d$, therefore $k\ge2$, and since each of the exponents of the variables $x_i$ appearing in $u$ is one, the first set of inequalities is satisfied. Moreover, $\sum_{i\in B_j}a_i=|B_j|<(d-1)k$ for all $j$, by the definition of $k$. Therefore, $u\in\omega_{R_\Gamma}$ as desired.
\end{proof}

By \cite[Proposition 1.9]{HMQ} we know that $\omega_{R_\Gamma}=\p_1\cap\dots\cap\p_r\cap Q_1\cap\dots\cap Q_n$, where $\mathcal{P}_\Gamma=\{\p_1,\dots,\p_r\}$ and $Q_i=({\bf x}_Ft:F\in\Gamma,i\in F)$, for $i=1,\dots,n$, are the height one monomial prime ideals of $R_\Gamma$ not containing $t$. If $(t)$ is a radical ideal, then this simplifies the computation of the canonical module of $R_\Gamma$. Indeed, if $(t)$ is radical, then $(t)=\p_1\cap\dots\cap\p_r$ and so $\omega_{R_\Gamma}=(t)\cap Q_1\cap\dots\cap Q_n$. We will prove that $(t)$ is radical if and only if $\dim(\Delta)=1$. We begin with

\begin{Proposition}\label{Lem:t-radical}
    Let $\Delta$ be a 1-dimensional unit-interval simplicial complex, and let $\Gamma=\Ind(\Delta)$. Then, $(t)\subset R_\Gamma$ is a radical ideal.
\end{Proposition}
\begin{proof}
Let $G$ to be the graph with $V(G)=V(\Delta)=[n]$ and $E(G)=\mathcal{F}(\Delta)$. Then $G$ is a proper interval graph. Let $B_1,\ldots,B_m$ be the maximal cliques of $G$ consisting of intervals. We may assume that for $i<j$, $\min B_i<\min B_j$.

Let $G_{\Gamma}$ be the graph whose edge set is $\mathcal{F}(\Gamma^{[1]})$, and let $C_1,\ldots,C_r$ be the minimal vertex covers of $G_{\Gamma}$. Since any proper interval graph is a perfect graph, by \cite[Theorem 1.10]{HMQ}, the minimal prime ideals of $(t)$ are $P_{C_i}=({\bf x}_{F}t:\ F\in \Gamma_{C_i})$ for $1\leq i\leq r$. Here $\Gamma_{C_i}$ is the subcomplex of $\Gamma$ induced on the set $C_i$. Notice that $G_{\Gamma}=G^c$. Moreover, $C_i=[n]\setminus F_i$, where $F_i$ is a maximal independent set of $G_{\Gamma}=G^c$ for $1\leq i\leq r$. Since maximal independent sets of $G^c$ are the maximal cliques of $G$ and any maximal clique of $G$ is an interval $B_i$, after relabeling the $C_i$'s, we obtain $C_i=[n]\setminus B_i$ for any $1\leq i\leq m$ and $r=m$.

We show that $(t)=\bigcap_{i=1}^m P_{C_i}$. Since each $P_{C_i}$ is a monomial ideal of $R_{\Gamma}$, the ideal $\bigcap_{i=1}^m P_{C_i}$ is a monomial ideal. So it is enough to show that any monomial in this intersection belongs to $(t)$. Let $u=x_1^{a_1}\cdots x_n^{a_n}t^k\in\bigcap_{i=1}^m P_{C_i}$.
The hyperplane $f_i(x)=-\sum_{j\notin C_i}x_j+x_{n+1}=0$ is the supporting hyperplane of the prime ideal $P_{C_i}$ for $1\leq i\leq m$, see the proof of \cite[Theorem 1.3]{HMQ}. Replacing $C_i$ by $[n]\setminus B_i$ we get $f_i(x)=-\sum_{j\in B_i}x_j+x_{n+1}$. So for the point $p_u=(a_1,\ldots,a_n,k)$ we have $f_i(p_u)=-\sum_{j\in B_i}a_j+k>0$ for all $i$. Hence, $\sum_{j\in B_i}a_j<k$ for all $i$. By induction on $k$ we show that $u\in (t)$. First assume that $k=1$. Then we have $\sum_{j\in B_i}a_j<1$ for all $i$. So $a_j=0$ for all $1\leq j\leq n$, since each $j$ belongs to some $B_i$. So $u=t\in (t)$.   By induction assume that for any monomial $v=x_1^{b_1}\cdots x_n^{b_n}t^{k-1}\in \bigcap_{i=1}^m P_{C_i}$, we have $v\in (t)$.
We claim that $u=({\bf x}_Ft)v$ for some monomials ${\bf x}_Ft\in R_{\Gamma}$ and $v\in \bigcap_{i=1}^m P_{C_i}$. Having this claim proved, it follows by induction that $v\in (t)$ and hence $u\in (t)$, as desired.

Set $d_i=\sum_{j\in B_i}a_j$. We have $d_i\leq k-1$ for all $i$. If $d_i\leq k-2$ for all $i$, then  $v=x_1^{a_1}\cdots x_n^{a_n}t^{k-1}\in \bigcap_{i=1}^m P_{C_i}$ and $u=tv\in (t)$. Now, assume that there exist intervals $B_i$ such that $d_i=k-1$. Let $i_1$ be the smallest integer such that $d_{i_1}=k-1$. We let $k_1$ be the largest integer in the interval $B_{i_1}$ such that $a_{k_1}>0$.
If for any $j>i_1$ with $k_1\notin B_{j}$ we have $d_{j}<k-1$, then we set $F=\{k_1\}$ and we obtain $u=({\bf x}_Ft)v$ with $v\in \bigcap_{i=1}^m P_{C_i}$, as desired.        
Otherwise, let $i_2>i_1$ be the smallest integer such that $k_1\notin B_{i_2}$ and $d_{i_2}=k-1$. Then $k_1<\min(B_{i_2})$. We let $k_2$ be the largest integer in the interval $B_{i_2}$ such that $a_{k_2}>0$. We claim that $\{k_1,k_2\}\in \Gamma$. If this is not the case, then $k_1,k_2\in B_{\ell}$ for some $\ell$. By the choice of $k_2$, and that $k_1\in B_{\ell}$ with $k_1<\min(B_{i_2})$, we have $\{j: j\in B_{i_2}, a_j>0\}\subseteq B_{\ell}$. Hence
$$
d_{\ell}=\sum_{j\in B_{\ell}}a_j\geq a_{k_1}+\sum_{j\in B_{i_2}}a_j=a_{k_1}+d_{i_2}\geq k,
$$
which is a contradiction.          
Hence, $\{k_1,k_2\}\in \Gamma$. 
Continuing this process, we obtain an independent set $F=\{k_1,\ldots,k_s\}\in \Gamma$ such that for any integer $j$ with $d_{j}=k-1$, $F\cap B_j\neq \emptyset$. Then we may write  $u=({\bf x}_Ft)v$, where $v=u/({\bf x}_Ft)$. By the choice of $F$, it follows that $v\in \bigcap_{i=1}^m P_{C_i}$, and the proof is complete. 
\end{proof}

As a consequence, we obtain
\begin{Corollary}\label{Cor:radical}
    Let $\Delta$ be a unit-interval simplicial complex, and let $\Gamma=\Ind(\Delta)$.  The ideal $(t)\subset R_\Gamma$ is radical if and only if $\dim(\Delta)=1$.
\end{Corollary}

\begin{proof}
    If $\dim(\Delta)=1$, then $(t)\subset R_\Gamma$ is radical by Proposition \ref{Lem:t-radical}.

    Conversely, assume that $(t)$ is a radical ideal, and suppose by contradiction that $\dim(\Delta)\ge2$. By our assumption we have $(t)=\p_1\cap\dots\cap\p_r$, where $\mathcal{P}_\Gamma=\{\p_1,\dots,\p_r\}$ is the set of the height one monomial prime ideals containing $t$. It follows that in $\textup{Cl}(R_\Gamma)$ we have
    \begin{equation}\label{eq:class(t)}
        \sum_{i=1}^{r}[\p_i]=[(t)]=0.
    \end{equation}
    Let $f_i(x)=\sum_{j=1}^{n+1}c_{i,j}x_j$ be the support form associated with $\p_i$ for $i=1,\dots,r$. It follows from \cite[Lemma 4.1(a)]{FHS} that the divisor class group $\textup{Cl}(R_\Gamma)$ is generated by the classes $[\p_1],\dots,[\p_r]$ with the unique relation $\sum_{i=1}^rc_{i,n+1}[\p_i]=0$. From this fact, it follows that equation (\ref{eq:class(t)}) holds if and only if $c_{1,n+1}=\dots=c_{r,n+1}$. However this is not possible by Corollary \ref{Cor:PGamma}. Indeed, the support form associated to any $P_i\in\mathcal{P}_\Gamma$ has the coefficient of $x_{n+1}$ equal to $1$, whereas the support form associated to any $L_j\in\mathcal{P}_\Gamma$ has the coefficient of $x_{n+1}$ equal to $\dim(\Delta)>1$.
    \end{proof}\bigskip

\section{Cohen-Macaulay sortable simplicial complexes}

In this section we give a characterization of the Cohen-Macaulay property of $d$-flag sortable simplicial complexes.
To do so, we consider more generally the independence complex of interval simplicial complexes, which are non-pure in general.\smallskip  

Let $\Delta$ be a simplicial complex on vertex set $[n]$. We say that $\Delta$ is an \textit{interval simplicial complex} if for any facet $F=\{i_1<\dots<i_{d}\}\in\mathcal{F}(\Delta)$, the interval $[i_1,i_{d}]$ is a clique of $\Delta$. Notice that $\Delta$ can be written as 
\begin{equation}\label{intervalComplex}
\Delta=\Delta_1^{[r_1]}\cup \Delta_2^{[r_2]}\cup \cdots\cup \Delta_m^{[r_m]}, 
\end{equation}
where $r_1,\ldots,r_n$ are positive integers, each $\Delta_j$ is a simplex and $\Delta_1^{[1]}\cup \Delta_2^{[1]}\cup \cdots\cup \Delta_m^{[1]}$ is a proper interval graph  with the intervals $B_j=V(\Delta_j)$, $j=1,\ldots,m$.\smallskip

It is clear that a unit-interval simplicial complexes is just an interval simplicial complex which is pure.

\begin{Theorem}\label{Thm:vd}
Let  $\Delta$ be an interval simplicial complex.
Then $\Ind(\Delta)$ is vertex decomposable.
\end{Theorem}

\begin{proof}
Let $\Delta$ be as described in (\ref{intervalComplex}), and set $\Gamma=\Ind(\Delta)$. We may assume that $V(\Delta)=[n]$, each $B_j\subset [n]$ is an interval and that $\min B_i<\min B_j$ for any $i<j$. 
We prove the assertion by double induction on $\sum_{j=1}^m r_j\geq m$ and $\sum_{j=1}^m |B_j|$.
If $\sum_{j=1}^m r_j=m$, then $r_j=1$ for all $j$. Hence, by assumption $\Delta$ is a proper interval graph, and so it is chordal. Then by \cite[Corollary 7]{W}, $\Ind(\Delta)$ is vertex decomposable.

Now, let $\sum_{j=1}^m r_j>m$ and assume inductively that for any interval simplicial complex $\Delta'=\Lambda_1^{[r'_1]}\cup\Lambda_2^{[r'_2]}\cup\cdots\cup\Lambda_m^{[r'_m]}$ with $\sum_{j=1}^m r'_j<\sum_{j=1}^m r_j$ or $\sum_{j=1}^m|V(K_j)|<\sum_{j=1}^m|V(B_j)|$, the simplicial complex
$\Ind(\Delta')$ is vertex decomposable. 

We have $\sum_{j=1}^m|B_j|\geq m$. If $\sum_{j=1}^m|B_j|=m$, then $|B_j|=1$ for all $j$. Since $r_j\geq 1$ for all $j$, it follows that $\Delta=\emptyset$ and $\Gamma=\langle [n]\rangle$ is a simplex. Hence, it is vertex decomposable.  
Similarly, if $|B_j|\leq r_j$ for all $j$, then $\Delta=\emptyset$ and $\Gamma=\langle [n]\rangle$ is a simplex, and so it is vertex decomposable. So we may assume that there exists an integer $1\leq j\leq m$ such that $|B_j|>r_j$. Let $p$ be the smallest such integer, and let $i=\max\{\ell: \ \ell\in B_p\}$. 

First, we show that $\Del_{\Gamma}(i)$ and $\Lk_{\Gamma}(i)$ are vertex decomposable.
Indeed, it is straightforward to see that $$\Del_{\Gamma}(i)=\Ind(\bigcup_{i\notin B_j}\Delta_j^{[r_j]})\cup \bigcup_{i\in B_j}(\Del_{\Delta_j}(i)^{[r_j]})$$ and $$\Lk_{\Gamma}(i)=\Ind(\bigcup_{i\notin B_j}\Delta_j^{[r_j]})\cup \bigcup_{i\in B_j}(\Del_{\Delta_j}(i)^{[r_j-1]}).$$ 
Notice that $\Del_{\Gamma}(i)$ and $\Lk_{\Gamma}(i)$ are interval simplicial complexes, as well.
Thus by our induction hypothesis $\Del_{\Gamma}(i)$ and $\Lk_{\Gamma}(i)$ are vertex decomposable. It remains to show that any facet of $\Del_{\Gamma}(i)$ is a facet of $\Gamma$. Let $F$ be a facet of $\Del_{\Gamma}(i)$. Two cases may happen.\medskip

\textsc{Case 1}. $B_p\setminus \{i\}\subseteq F$. Then $B_p\subseteq F\cup \{i\}$. Since $|B_p|>r_p$, it follows that $F\cup \{i\}$ contains a facet $H$ of $\Delta_p^{[r_p]}$, which is a facet of $\Delta$. Hence,  $F\cup \{i\}\notin \Gamma$. This shows that $F$ is a facet of $\Gamma$.\medskip

\textsc{Case 2}. $B_p\setminus \{i\}\nsubseteq F$. Then there exists $t\in B_p\setminus F$ with $t<i$. Since $F$ is a facet of $\Del_{\Gamma}(i)$ and $t\neq i$, we have $F\cup\{t\}\notin \Gamma$. So there exists a facet $H$ of $\Delta$ such that $H\subseteq F\cup\{t\}$. Then $H$ is a facet of $\Delta_q^{[r_q]}$ for some $q$ and $t\in H$. Since $|B_j|\leq r_j$ for $j=1,\ldots,p-1$, we have $\Delta_j^{[r_j]}=\emptyset$ for $j=1,\ldots,p-1$. Hence $q\geq p$. From $t,i\in B_p$, $t\in B_q$ and the inequalities $t<i$ and $p\leq q$, we obtain $i\in B_q$. Thus $H'=(H\setminus \{t\})\cup\{i\}$ is a facet of $\Delta_q^{[r_q]}$ and hence a facet of $\Delta$. Moreover, $H'\subseteq F\cup\{i\}$.
So $F\cup \{i\}\notin \Gamma$, which means that $F$ is a facet of $\Gamma$.
\end{proof}

Notice that for  $\Gamma=\Ind(\Delta)$, the Stanley-Reisner ideal of $\Gamma$ is the facet ideal $I(\Delta)$ of $\Delta$. Therefore,
applying Theorem \ref{Thm:vd} together with \cite[Theorem 2.3]{MK} and \cite[Theorem 3.1]{M} we have

\begin{Corollary}\label{Cor:vd}
Let  $\Delta$ be an interval simplicial complex.
Then
\begin{enumerate}
\item[\textup{(a)}] $S/I(\Delta)$ is Cohen-Macaulay if and only if $I(\Delta)$ is unmixed. 
\item[\textup{(b)}] $I(\Delta)^{\vee}$ is vertex splittable and hence it has linear quotients.
\item[\textup{(c)}] $\mathcal{R}(I(\Delta)^{\vee})$ is normal Cohen-Macaulay and $I(\Delta)^{\vee}$ satisfies the strong persistence property.
\item [\textup{(d)}] If $I(\Delta)^{\vee}$ is equigenerated, then the toric ring $K[u: u\in \mathcal{G}(I(\Delta)^{\vee})]$ is normal and Cohen-Macaulay. 
\end{enumerate}
\end{Corollary}

\begin{Corollary}
Any $d$-flag sortable simplicial complex $\Gamma$ is vertex decomposable. In particular, $\Gamma$ is Cohen-Macaulay if and only if it is pure.
\end{Corollary}

\begin{proof}
By Theorem~\ref{Thm:Ind}, we have $\Gamma=\Ind(\Delta)$, where $\Delta$ is a unit-interval simplicial complex. Thus by Theorem~\ref{Thm:vd}, $\Gamma$ is vertex decomposable.
\end{proof}

\noindent\textbf{Acknowledgment.}
A. Ficarra was partly supported by INDAM (Istituto Nazionale di Alta Matematica), and also by the Grant JDC2023-051705-I funded by
MICIU/AEI/10.13039/501100011033 and by the FSE+. S. Moradi is supported by the Alexander von Humboldt Foundation.

\end{document}